\documentclass{amsart}
\usepackage[all]{xy}
\usepackage{color}
\usepackage{amsthm}
\usepackage{amssymb}
\usepackage[colorlinks=true]{hyperref}

\numberwithin{equation}{section}

\newtheorem{theorem}[equation]{Theorem} 
\newtheorem*{theorem*}{Theorem}
\newtheorem{lemma}[equation]{Lemma}
\newtheorem{proposition}[equation]{Proposition}
\newtheorem{corollary}[equation]{Corollary}
\newtheorem*{corollary*}{Corollary}

\theoremstyle{remark}

\theoremstyle{remark}
\newtheorem{remark}[equation]{Remark}

\newcommand{\cA}{{\mathcal A}}
\newcommand{\cB}{{\mathcal B}}
\newcommand{\cC}{{\mathcal C}}

\newcommand{\cN}{{\mathcal N}}
\newcommand{\cO}{{\mathcal O}}

\newcommand{\cS}{{\mathcal S}}

\newcommand{\bbC}{\mathbb{C}}

\newcommand{\bbQ}{\mathbb{Q}}
\newcommand{\bbZ}{\mathbb{Z}}

\DeclareMathOperator{\SmProj}{SmProj} 

\DeclareMathOperator{\id}{id}

\DeclareMathOperator{\NChow}{NChow} 
\DeclareMathOperator{\NNum}{NNum} 
\DeclareMathOperator{\Chow}{Chow} 
\DeclareMathOperator{\Num}{Num} 

\newcommand{\Ab}{\mathrm{Ab}}

\newcommand{\NHomo}{\mathrm{NHomo}}

\newcommand{\Curv}{\mathrm{Curv}}

\newcommand{\perf}{\mathsf{perf}}

\newcommand{\Hom}{\mathsf{Hom}}
\newcommand{\Homo}{\mathrm{Homo}}
\newcommand{\End}{\mathsf{End}}

\newcommand{\op}{\mathsf{op}}

\newcommand{\too}{\longrightarrow}

\newcommand{\ie}{\textsl{i.e.}\ }
\newcommand{\eg}{\textsl{e.g.}}

\begin{document}

\title[Jacobians of noncommutative motives]{Jacobians of noncommutative motives}

\author{matilde Marcolli and Gon{\c c}alo~Tabuada}

\address{Matilde Marcolli, Mathematics Department, Mail Code 253-37, Caltech, 1200 E.~California Blvd. Pasadena, CA 91125, USA}
\email{matilde@caltech.edu} 
\urladdr{http://www.its.caltech.edu/~matilde}

\address{Gon{\c c}alo Tabuada, Department of Mathematics, MIT, Cambridge, MA 02139, USA}
\email{tabuada@math.mit.edu}
\urladdr{http://math.mit.edu/~tabuada}

\thanks{The first named author was partially supported by the NSF grants DMS-0901221, DMS-1007207, DMS-1201512 and PHY-1205440. The second named author was partially supported by the NEC Award-2742738.}
\date{\today}


\abstract{In this article one extends the classical theory of (intermediate) Jacobians to the ``noncommutative world''. Concretely, one constructs a $\bbQ$-linear additive Jacobian functor $N \mapsto {\bf J}(N)$ from the category of noncommutative Chow motives to the category of abelian varieties up to isogeny, with the following properties: (i) the first de Rham cohomology group of ${\bf J}(N)$ agrees with the subspace of the odd periodic cyclic homology of $N$ which is generated by algebraic curves; (ii) the abelian variety ${\bf J}(\perf(X))$ (associated to the derived dg category $\perf(X)$ of a smooth projective $k$-scheme $X$) identifies with the union of all the intermediate algebraic Jacobians of $X$.
}}
\maketitle 
\vskip-\baselineskip
\vskip-\baselineskip

\section{Introduction}\label{sec:introduction}
\subsection*{Jacobians} 
The {\em Jacobian} $J(C)$ of a curve $C$ was introduced by Weil \cite{Weil} in the forties as a geometric replacement for the first cohomology group $H^1(C)$ of $C$. This construction was latter generalized to the {\em Picard} $\mathrm{Pic}^0(X)$ and the {\em Albanese} $\mathrm{Alb}(X)$ varieties of a smooth projective $k$-scheme of dimension $d$. When $X=C$ one has $\mathrm{Pic}^0(C)=\mathrm{Alb}(C)=J(C)$, but in general $\mathrm{Pic}^0(X)$ (resp. $\mathrm{Alb}(X)$) is a geometric replacement for $H^1(X)$ (resp. for $H^{2d-1}(X)$).

In the case where $k$ is an algebraically closed subfield of $\bbC$, Griffiths \cite{Griffiths} extended these constructions to a whole family of Jacobians. Concretely, the $i^{\mathrm{th}}$ Jacobian $J_i(X)$ of $X$ is the compact torus
\begin{eqnarray*}
J_i(X):= \frac{H_B^{2i+1}(X,\bbC)}{F^{i+1}H_B^{2i+1}(X,\bbC) + H_B^{2i+1}(X,\bbZ)} && 0 \leq i \leq d-1\,,
\end{eqnarray*}
where $H_B$ stands for Betti cohomology and $F$ for the Hodge filtration. In contrast with $J_0(X) = \mathrm{Pic}^0(X)$ and $J_{d-1}(X)= \mathrm{Alb}(X)$, the intermediate Jacobians are {\em not} algebraic. Nevertheless, they contain an algebraic variety $J_i^a(X)\subseteq J_i(X)$ defined by the image of the Abel-Jacobi map
\begin{eqnarray}\label{eq:AbelJacobi}
AJ_i: CH^{i+1}(X)_\bbZ^{\mathrm{alg}} \to J_i(X) && 0 \leq i \leq d-1\,,
\end{eqnarray}
where $CH^{i+1}(X)_\bbZ^{\mathrm{alg}}$ stands for the group of algebraically trivial cycles of codimension $i+1$; consult Vial \cite[page~12]{Vial} for further details. When $i=0, d-1$ the map \eqref{eq:AbelJacobi} is surjective and so $J_0^a(X)=\mathrm{Pic}^0(X)$ and $J_{d-1}^a(X)=\mathrm{Alb}(X)$.

Note that all these abelian varieties are only well-defined up to isogeny. In what follows one will write $\Ab(k)_\bbQ$ for the category of abelian varieties up to isogeny; consult Deligne~\cite[page~4]{Deligne} for further details.
\subsection*{Motivating question}
All the above classical constructions in the ``commutative world'' lead us naturally to the following motivating question:

\vspace{0.2cm}

{\it Question: Can the theory of (intermediate) Jacobians be extended to the ``noncommutative world''~?}
\subsection*{Statement of results}
Let $k$ be a field of characteristic zero. Recall from \S\ref{sub:NCmotives} the construction of the category $\NChow(k)_\bbQ$ of noncommutative Chow motives (with rational coefficients). Examples of noncommutative Chow motives include finite dimensional $k$-algebras of finite global dimension (\eg\ path algebras of finite quivers without oriented loops) as well as derived dg categories of perfect complexes $\perf(X)$ of smooth projective $k$-schemes\footnote{Or more generally smooth and proper Deligne-Mumford stacks.} $X$; consult also Kontsevich \cite{IAS} for examples coming from deformation quantization.

As proved in \cite[Thm.~7.2]{Galois}, periodic cyclic homology gives rise to a well-defined $\otimes$-functor with values in the category of finite dimensional super $k$-vector spaces
\begin{equation}\label{eq:HP}
HP^\pm: \NChow(k)_\bbQ \too \mathrm{sVect}(k)\,.
\end{equation}
Given a noncommutative Chow motive $N$, let us denote by $HP_{\mathrm{curv}}^-(N)$ the piece of $HP^-(N)$ which is {\em generated by curves}, \ie the $k$-vector space
$$HP^-_{\mathrm{curv}}(N):=\sum_{C,\Gamma} \mathrm{Im} (HP^-(\perf(C)) \stackrel{HP^-(\Gamma)}{\too} HP^-(N))\,,$$
where $C$ is smooth projective curve and $\Gamma:\perf(C) \to N$ a morphism in $\NChow(k)_\bbQ$.

Inspired by Grothendieck's standard conjecture $D$ (see \cite[\S5.4.1]{Andre}), the authors have introduced at \cite[page~4]{Galois} the noncommutative standard conjecture $D_{NC}$. Given a noncommutative Chow motive $N$, $D_{NC}(N)$ claims that the homological and the numerical equivalence relations on the rationalized Grothendieck group $K_0(N)_\bbQ$ agree.  Our first main result is the following: 
\begin{theorem}\label{thm:main}
\begin{itemize}
\item[(i)] There is a well-defined $\bbQ$-linear additive Jacobian functor
\begin{eqnarray}\label{eq:J-functor}
\NChow(k)_\bbQ \too \Ab(k)_\bbQ && N \mapsto {\bf J}(N)\,.
\end{eqnarray}
\item[(ii)] For every $N \in \NChow(k)_\bbQ$, there exists a smooth projective curve $C_N$ and a morphism $\Gamma_N: \perf(C_N) \to N$ such that $H^1_{dR}({\bf J}(N))=\mathrm{Im}HP^-(\Gamma_N)$, where $H_{dR}$ stands for de Rham cohomology. Consequently, one has an inclusion of $k$-vector spaces $H^1_{dR}({\bf J}(N)) \subseteq HP^-_{\mathrm{curv}}(N)$.
\item[(iii)] Given a noncommutative Chow motive $N$, assume that the noncommutative standard conjecture $D_{NC}(\perf(C) \otimes N)$ holds for every smooth projective curve $C$. Under such assumption the inclusion of item (ii) becomes an equality
\begin{equation}\label{eq:H1}
H^1_{dR}({\bf J}(N)) = HP_{\mathrm{curv}}^-(N)\,.
\end{equation}
\end{itemize}
\end{theorem}
Since the dimension of any abelian variety $A$ is equal to the dimension of $H^1_{dR}(A)$, one concludes from item (ii) that the dimension of ${\bf J}(N)$ is always bounded by the dimension of the $k$-vector space $HP^-_{\mathrm{curv}}(N)$. As proved in \cite[Thm.~1.5]{Galois}, the implication $D(X) \Rightarrow D_{NC}(\perf(X))$ holds for every smooth projective $k$-scheme $X$. Moreover, $\perf(C) \otimes \perf(X) \simeq \perf(C\times X)$; see \cite[Prop.~6.2]{regular}. Hence, when $N=\perf(X)$, the assumption of item (iii) follows from Grothendieck standard conjecture $D(C \times X)$ which is known to be true when $X$ is of dimension $\leq 4$; see \cite[\S5.4.1.4]{Andre}. Note also that equality \eqref{eq:H1} describes all the de Rham cohomology of ${\bf J}(N)$ since for every abelian variety one has $H^i_{dR}({\bf J}(N)) \simeq \wedge^iH^1_{dR}({\bf J}(N))$; see \cite[\S4.3.3]{Andre}. Intuitively speaking, the abelian variety ${\bf J}(N)$ is a geometric replacement for the $k$-vector space $HP^-_{\mathrm{curv}}(N)$.

Now, recall from \S\ref{sub:motives} that one has a classical contravariant $\otimes$-functor $M(-)$ from the category $\SmProj(k)$ of smooth projective $k$-schemes to the category $\Chow(k)_\bbQ$ of Chow motives (with rational coefficients). As explained in \cite[Prop.~4.2.5.1]{Andre}, de Rham cohomology factors through $\Chow(k)_\bbQ$. Hence, given a smooth projective $k$-scheme $X$ of dimension $d$, one can proceed as above and define 
\begin{eqnarray*}
NH_{dR}^{2i+1}(X):= \sum_{C,\gamma_i} \mathrm{Im} \big(H^1_{dR}(C) \stackrel{H^1_{dR}(\gamma_i)}{\too} H^{2i+1}_{dR}(X) \big) && 0 \leq i \leq d-1\,,
\end{eqnarray*}
where now $\gamma_i: M(C) \to M(X)(i)$ is a morphism in $\Chow(k)_\bbQ$. By restricting the intersection bilinear pairings on de Rham cohomology (see \cite[\S3.3]{Andre}) to these pieces one obtains
\begin{eqnarray}\label{eq:pairings1}
\langle-,- \rangle : NH_{dR}^{2d-2i-1}(X) \times NH_{dR}^{2i+1}(X) \too k && 0 \leq i \leq d-1\,.
\end{eqnarray}
Our second main result is the following:
\begin{theorem}\label{thm:main2}
Let $k$ be an algebraically closed subfield of $\bbC$ and $X$ be a smooth projective $k$-scheme of dimension $d$. Assume that the pairings \eqref{eq:pairings1} are non-degenerate. Under such assumption there is an isomorphism of abelian varieties up to isogeny
\begin{equation}\label{eq:isom-last}
 {\bf J}(\perf(X)) \simeq \cup_{i=0}^{d-1} J^a_i(X)\,.
\end{equation}
Moreover, $H^1_{dR}({\bf J}(\perf(X)))\otimes_k \bbC \simeq \oplus_{i=0}^{d-1}NH_{dR}^{2i+1}(X) \otimes_k \bbC$.
\end{theorem}
As explained in Remark~\ref{rk:explanation}, the pairings \eqref{eq:pairings1} with $i=0$ and $i=d-1$ are always non-degenerate. Moreover, the non-degeneracy of the remaining cases follows from Grothendieck's standard conjecture of Lefschetz type; see \cite[\S5.2.4]{Andre}. Hence, the above pairings \eqref{eq:pairings1} are non-degenerate for curves, surfaces, abelian varieties, complete intersections, uniruled threefolds, rationally connected fourfolds, and for any smooth hypersurface section, product, or finite quotient thereof (and if one trusts Grothendieck they are non-degenerate for all smooth projective $k$-schemes). As a consequence one obtains the following: 
\begin{corollary}\label{cor:main}
Let $k$ be an algebraically closed subfield of $\bbC$. For every smooth projective curve $C$ (resp. surface $S$) there is an isomorphism of abelian varieties up to isogeny ${\bf J}(\perf(C))\simeq J(C)$ (resp. ${\bf J}(\perf(S))\simeq \mathrm{Pic}^0(S) \cup \mathrm{Alb}(S)$).
\end{corollary}
Note that since the Picard and the Albanese varieties are isogenous (see \cite[\S4.3.4]{Andre}), one can replace $\mathrm{Pic}^0(S)$ in the above isomorphism by $\mathrm{Alb}(S)$ or vice-versa.

Theorems~\ref{thm:main} and \ref{thm:main2} (and Corollary~\ref{cor:main}) provide us an affirmative answer to our motivating question. Roughly speaking, the classical theory of Jacobians can in fact be extended to the ``noncommutative world'' as long as one works with all the intermediate Jacobians simultaneously. Note that this restriction is an intrinsic feature of the ``noncommutative world'' which cannot be avoid because as soon as one passes from $X$ to $\perf(X)$ one loses track of the individual pieces of $H^\ast_{dR}(X)$.

\medbreak\noindent\textbf{Acknowledgments:} The authors are very grateful to Joseph Ayoub, Dmitry Kaledin and Burt Totaro for useful discussions.
\section{Preliminaries}
Throughout this note one will reserve the letter $k$ for the base field (which will assumed of characteristic zero) and the symbol $(-)^\natural$ for the classical pseudo-abelian construction. The category of smooth projective $k$-schemes will be denoted by $\SmProj(k)$, its full subcategory of smooth projective curves by $\Curv(k)$, and the de Rham (resp. Betti when $k \subset \bbC$) cohomology functor by
\begin{eqnarray}\label{eq:deRham}
& H^\ast_{dR}: \SmProj(k)^\op \too \mathrm{GrVect}(k) & H^\ast_{B}: \SmProj(k)^\op \too \mathrm{GrVect}(\bbQ)\,,
\end{eqnarray}
where $\mathrm{GrVect}(k)$ (resp. $\mathrm{GrVect}(\bbQ)$) stands for the category of finite dimensional $\bbZ$-graded $k$-vector spaces (resp. $\bbQ$-vector spaces); consult \cite[\S3.4.1]{Andre} for details.
\subsection{Motives}\label{sub:motives}
One will assume that the reader has some familiarity with the category $\Chow(k)_\bbQ$ of Chow motives (see \cite[\S4]{Andre}), with the category $\Homo(k)_\bbQ$ of homological motives (see \cite[\S4.4.2]{Andre}), and with the category $\Num(k)_\bbQ$ of numerical motives (see \cite[\S4.4.2]{Andre}). The Tate motive will be denoted by $\bbQ(1)$. Recall that by construction one has a sequence of $\otimes$-functors
$$ \SmProj(k)^\op \stackrel{M(-)}{\too} \Chow(k)_\bbQ \too \Homo(k)_\bbQ \too \Num(k)_\bbQ\,.$$
\subsection{Dg categories}\label{sub:dg}
For a survey article on dg categories one invites the reader to consult Keller's ICM address \cite{ICM-Keller}. Recall from Kontsevich \cite{IAS,ENS,Miami,finMot} that a dg category $\cA$ is called {\em smooth} and {\em proper} if it is perfect as a $\cA\text{-}\cA$-bimodule over itself and if all its Hom complexes of $k$-vector spaces have finite total cohomology. Examples include ordinary finite dimensional $k$-algebras of finite global dimension and the (unique) dg enhancements $\perf(X)$ of the derived categories of perfect complexes of $\cO_X$-modules in the case where $X \in\SmProj(k)$; consult Lunts-Orlov \cite{LO} and \cite[Example~5.5]{CT1} for further details.
\subsection{Noncommutative motives}\label{sub:NCmotives}
In this subsection one recalls the construction of the categories of noncommutative pure motives. For further details one invites the reader to consult the survey article \cite{survey}.

The category $\NChow(k)_\bbQ$ of {\em noncommutative Chow motives} is the pseudo-abelian envelope of the category whose objects are the smooth and proper dg categories, whose morphisms from $\cA$ to $\cB$ are given by the $\bbQ$-linearized Grothendieck group $K_0(\cA^\op \otimes \cB)_\bbQ$, and whose composition law is induced by the (derived) tensor product of bimodules.

The category $\NHomo(k)_\bbQ$ of {\em noncommutative homological motives} is the pseudo-abelian envelope of the quotient category $\NChow(k)_\bbQ/\mathrm{Ker}(HP^\pm)$, where $HP^\pm$ is the above functor \eqref{eq:HP} induced by periodic cyclic homology.

The category $\NNum(k)_\bbQ$ of {\em noncommutative numerical motives} is the pseudo-abelian envelope of the quotient category $\NChow(k)_\bbQ/\cN$, where $\cN$ is the largest $\otimes$-ideal\footnote{As proved in \cite{Numerical,Kontsevich}, this ideal admits two explicit descriptions: one in terms of Hochschild homology and the other one in terms of a well-behaved bilinear form on the Grothendieck group.} of $\NChow(k)_\bbQ$ distinct from the entire category. 

All the above categories carry a symmetric monoidal structure which is induced by the tensor product of dg categories. Moreover, as in the case of pure motives, one has also a sequence of (full) $\otimes$-functors
$$ \NChow(k)_\bbQ \too \NHomo(k)_\bbQ \too \NNum(k)_\bbQ\,.$$
\subsection{Orbit categories}\label{sub:orbit}
Let $\cC$ be a $\bbQ$-linear symmetric monoidal category and $\cO \in \cC$ a $\otimes$-invertible object. As explained in \cite[\S7]{CvsNC}, the {\em orbit category} $\cC/_{\!\!-\otimes \cO}$ has the same objects as $\cC$ and morphisms given by 
$$ \Hom_{\cC/_{\!\!-\otimes \cO}}(a,b) := \oplus_{j \in \bbZ} \Hom_\cC(a,b\otimes \cO^{\otimes j})\,.$$
The composition law is induced by the one on $\cC$. By construction, $\cC/_{\!\!-\otimes \cO}$ is $\bbQ$-linear and symmetric monoidal (see \cite[Lemma~7.1]{CvsNC}) and comes equipped with a canonical projection $\bbQ$-linear $\otimes$-functor $\mu:\cC\to \cC/_{\!\!-\otimes \cO}$. Moreover, $\mu$ is endowed with a canonical $2$-isomorphism $\mu \circ (-\otimes \cO) \stackrel{\sim}{\Rightarrow} \mu$ and is $2$-universal among all such functors.
\section{A key bridge}\label{sec:bridge}
In this section one describes a precise bridge between the categories of motives and the categories of noncommutative motives. This bridge will play a key role in the construction of the Jacobian functor. 
\begin{theorem}\label{thm:bridge}
There exist $\bbQ$-linear additive $\otimes$-functors $R, R_H$ and $R_\cN$ making the following diagram commute
\begin{equation}\label{eq:dig-1}
\xymatrix{
\Chow(k)_\bbQ \ar[d] \ar[r]^-\mu & \Chow(k)_\bbQ/_{\!\!-\otimes \bbQ(1)} \ar[d] \ar[r]^-R & \NChow(k)_\bbQ \ar[d] \\
\Homo(k)_\bbQ\ar[r]^-\mu \ar[d] & \Homo(k)_\bbQ/_{\!\!-\otimes \bbQ(1)} \ar[d] \ar[r]^-{R_H} & \NHomo(k)_\bbQ \ar[d] \\
\Num(k)_\bbQ \ar[r]^-\mu & \Num(k)_\bbQ/_{\!\!-\otimes \bbQ(1)} \ar[r]^-{R_\cN} & \NNum(k)_\bbQ\,.
}
\end{equation}
Moreover, $R_H$ is full and $R$ and $R_\cN$ are fully faithful.
\end{theorem}
\begin{proof}
Note first that the middle vertical functors are induced by the left vertical ones. Consequently, the upper left square and the lower left square are commutative. The $\bbQ$-linear additive fully faithful $\otimes$-functors $R$ and $R_\cN$, making the outer square of \eqref{eq:dig-1} commutative, where constructed in \cite[Thm.~1.12]{Numerical}. Let us now construct $R_H$. Consider the composed functor
\begin{equation}\label{eq:comp}
\Chow(k)_\bbQ \stackrel{\mu}{\too} \Chow(k)_\bbQ/_{\!\!-\otimes \bbQ(1)} \stackrel{R}{\too} \NChow(k)_\bbQ \stackrel{HP^\pm}{\too} \mathrm{sVect}(k)\,.
\end{equation}
As explained in the proof of \cite[Thm.~1.3]{Galois}, this composed functor agrees with the super-perioditization of de Rham cohomology \eqref{eq:deRham}, \ie it agrees with 
\begin{eqnarray}\label{eq:description}
&\Chow(k)_\bbQ \too \mathrm{sVect}(k)& X \mapsto \big(\oplus_{i \,\mathrm{even}} H^i_{dR}(X), \oplus_{i \,\mathrm{odd}} H^i_{dR}(X)  \big)\,.
\end{eqnarray}
As a consequence, the composition
$$ \Chow(k)_\bbQ \stackrel{\mu}{\too} \Chow(k)_\bbQ/_{\!\!-\otimes \bbQ(1)} \stackrel{R}{\too} \NChow(k)_\bbQ \too \NHomo(k)_\bbQ$$
descends to $\Homo(k)_\bbQ$. Moreover, since the Tate motive $\bbQ(1)$ is mapped to the $\otimes$-unit of $\NHomo(k)_\bbQ$, one obtains by the universal property of the orbit category a well-defined $\bbQ$-linear additive $\otimes$-functor $R_H$ making the upper right square commute. It remains then only to show that the lower right square is also commutative. Note that the lower rectangle (consisting of the lower left and right squares) is commutative. This follows from the fact the outer square of \eqref{eq:dig-1} is commutative, that the functor $\Chow(k)_\bbQ \to \Homo(k)_\bbQ$ is full, and that every object in $\Homo(k)_\bbQ$ is a direct factor of an object in the image of this latter functor. Consequently, one observes that the two composed functors from $\Homo(k)_\bbQ/_{\!\!-\otimes \bbQ(1)}$ to $\NNum(k)_\bbQ$ agree when precomposed with $\mu$. Using once again the universal property of the orbit category, one concludes then that these two composed functors are in fact the same, \ie that the lower right square is commutative.

Let us now prove that $R_H$ is full. Consider the following commutative diagram
$$
\xymatrix{
\Chow(k)_\bbQ/_{\!\!-\otimes \bbQ(1)} \ar[d] \ar[r]^-R & \NChow(k)_\bbQ \ar[d] \ar[r]^-{HP^\pm} & \mathrm{sVect}(k) \\
(\Chow(k)_\bbQ/_{\!\!-\otimes \bbQ(1)})/\mathrm{Ker} \ar[r]^-{\overline{R}} & \NHomo(k)_\bbQ \ar[ur]_-{HP^\pm} & \,,
}
$$
where $\mathrm{Ker}$ stands for the kernel of the upper horizontal composition and $\overline{R}$ for the induced functor. Clearly, $\overline{R}$ is fully faithful since this is the case for $R$. By the universal property of the orbit category, one observes that the functor $R_H$ admits the following factorization
$$ R_H: \Homo(k)_\bbQ/_{\!\!-\otimes \bbQ(1)} \to (\Chow(k)_\bbQ/_{\!\!-\otimes \bbQ(1)})/\mathrm{Ker} \stackrel{\overline{R}}{\too} \NHomo(k)_\bbQ\,.$$
Hence, it suffices to show that the left-hand-side functor is full. This is the case since every morphism $[\{f_j\}_{j \in \bbZ}]$ in $(\Chow(k)_\bbQ/_{\!\!-\otimes \bbQ(1)})/\mathrm{Ker}$ admits a canonical lift to a morphism $\{[f_j]\}_{j \in \bbZ}$ in $\Homo(k)_\bbQ$.
\end{proof}
\section{Proof of theorem \ref{thm:main}}
{\bf Item (i)} As explained in \cite[Prop.~4.2.5.1]{Andre}, de Rham cohomology descends (uniquely) to $\Chow(k)_\bbQ$ and hence to $\Homo(k)_\bbQ$. One obtains then the following commutative diagram:
\begin{equation}\label{eq:diagram}
\xymatrix{
\SmProj(k)^\op \ar[d]_-{M(-)} \ar[rr]^-{H^\ast_{dR}} && \mathrm{GrVect}(k) \\
\Homo(k)_\bbQ \ar@/_2ex/[urr]_-{H^\ast_{dR}} & & \,.
}
\end{equation}
Recall from \cite[\S5.1]{Andre} that for every $X \in \SmProj(k)$ of dimension $d$ one has well-defined K{\"u}nneth projectors
\begin{eqnarray*}
\pi^i: H^\ast_{dR}(X) \twoheadrightarrow H^i_{dR}(X) \hookrightarrow H^\ast_{dR}(X) && 0 \leq i \leq 2d\,.
\end{eqnarray*}
As explained in \cite[\S4.3.4]{Andre}, the first K{\"u}nneth projector $\pi^1$ is always {\em algebraic}, \ie there exists a (unique) correspondence $\underline{\pi}^1 \in \End_{\Homo(k)_\bbQ}(M(X))$ such that $H^\ast_{dR}(\underline{\pi}^1)=\pi^1$. Moreover, as proved in \cite[Corollary~3.4]{Scholl}, the passage from a smooth projective curve to its Jacobian (abelian) variety $J(C)$ gives rise to an equivalence of categories
\begin{eqnarray}\label{eq:equivalence}
\Homo(k)_\bbQ \supset \{\underline{\pi}^1 M(C) \, |\, C \in \Curv(k)\}^\natural \stackrel{\simeq}{\too} \Ab(k)_\bbQ && C \mapsto J(C)\,.
\end{eqnarray}
This equivalence of categories is independent of the equivalence relation on cycles and so the diagram commutes:
\begin{equation}\label{eq:diag-auxiliar}
\xymatrix{
\Homo(k)_\bbQ \supset \{\underline{\pi}^1 M(C) \, |\, C \in \Curv(k)\}^\natural  \ar[r]^-\simeq \ar@<-15ex>[d] \ar@<5ex>@{=}[d] & \Ab(k)_\bbQ \ar@{=}[d] \\
\Num(k)_\bbQ \supset \{\underline{\pi}^1 M(C) \, |\, C \in \Curv(k)\}^\natural \ar[r]^-\simeq & \Ab(k)_\bbQ\,.
}
\end{equation}
By combining \cite[Thm.~1.12]{Numerical} with Theorem~\ref{thm:bridge}, one obtains the following commutative diagram
\begin{equation}\label{eq:big}
\xymatrix@C=2em@R=2em{
& \SmProj(k)^\op \ar[d]_-{M(-)} \ar@/^2ex/[ddrr]^-{\perf(-)} &&\\
& \Chow(k)_\bbQ \ar[d] \ar[dr]^-\mu && \\
& \Homo(k)_\bbQ  \ar[d]_\mu \ar[dl] & \Chow(k)_\bbQ/_{\!\!-\otimes \bbQ(1)} \ar[dl] \ar[r]^-R & \NChow(k)_\bbQ\ar[d]  \\
\Num(k)_\bbQ \ar[d]_\mu & \Homo(k)_\bbQ/_{\!\!-\otimes \bbQ(1)}  \ar[dl]  \ar[rr]^-{R_H} && \ar[dl] \NHomo(k)_\bbQ  \\
\Num(k)_\bbQ/_{\!\!-\otimes \bbQ(1)} \ar[rr]_{R_{\cN}} & & \NNum(k)_\bbQ &
}
\end{equation}
and consequently, using \eqref{eq:diag-auxiliar}, the diagram
\begin{equation}\label{eq:composed10}
\xymatrix{
\Ab(k)_\bbQ \ar[r]^-{\eqref{eq:equivalence}^{-1}} \ar@{=}[d] & \Homo(k)_\bbQ \ar[r]^-\mu \ar[d] & \Homo(k)_\bbQ/_{\!\!-\otimes \bbQ(1)} \ar[r]^-{R_H} \ar[d] & \NHomo(k)_\bbQ \ar[d] \\
\Ab(k)_\bbQ \ar[r]^-{\eqref{eq:equivalence}^{-1}} & \Num(k)_\bbQ \ar[r]^-\mu & \Num(k)_\bbQ/_{\!\!-\otimes \bbQ(1)} \ar[r]^-{R_\cN} & \NNum(k)_\bbQ\,. 
}
\end{equation}
As proved in Lemma~\ref{lem:aux} below, both horizontal compositions in \eqref{eq:composed10} are fully faithful. Let us then write $\overline{\Ab}(k)_\bbQ$ for their image. Note that by construction one has the commutative square:
\begin{equation}\label{eq:aux33}
\xymatrix{
\Ab(k)_\bbQ \ar[rrr]^-{R_H \circ \mu \circ \eqref{eq:equivalence}^{-1}}_-{\simeq}  \ar@{=}[d]  &&& \overline{\Ab}(k)_\bbQ \subset \NHomo(k)_\bbQ \ar@<-8ex>@{=}[d]  \ar@<4ex>[d]\\
\Ab(k)_\bbQ \ar[rrr]^-{R_\cN \circ \mu \circ \eqref{eq:equivalence}^{-1}}_-{\simeq} &&& \overline{\Ab}(k)_\bbQ \subset \NNum(k)_\bbQ \,.
}
\end{equation}
Now, as proved in \cite[Thm.~1.9(ii)]{Numerical}, the category $\NNum(k)_\bbQ$ is abelian semi-simple. As a consequence every object $N \in \NNum(k)_\bbQ$ admits a unique finite direct sum decomposition $N \simeq S_1 \oplus \cdots \oplus S_n$ into simple objects. Let us denote by $\cS$ the set of those simple objects which belong to $\overline{\Ab}(k)_\bbQ \subset \NNum(k)_\bbQ$. Making use of it, one introduces the truncation functor
\begin{eqnarray}\label{eq:truncation}
\NNum(k)_\bbQ \too \overline{\Ab}(k)_\bbQ && N \mapsto \tau(N)
\end{eqnarray}
that associates to every noncommutative numerical motive $N=S_1 \oplus \cdots \oplus S_n$ its subsume consisting of those simple objects that belong to $\cS$. Note that \eqref{eq:truncation} is well defined since the equality $\Hom_{\NNum(k)_\bbQ}(S_i,S_j)=\delta_{ij} \cdot \bbQ$ implies that every morphism $N \to N'$ in $\NNum(k)_\bbQ$ restricts to a morphism $\tau(N) \to \tau(N')$ in $\overline{\Ab}(k)_\bbQ$. The desired Jacobian functor \eqref{eq:J-functor} can now be defined as the following composition
$$ {\bf J}(-): \NChow(k)_\bbQ \too \NNum(k)_\bbQ \stackrel{\eqref{eq:truncation}}{\too} \overline{\Ab}(k)_\bbQ \stackrel{(R_\cN \circ \mu \circ \eqref{eq:equivalence}^{-1})^{-1}}{\too} \Ab(k)_\bbQ\,.$$
Clearly, this functor is $\bbQ$-linear and additive.
\begin{lemma}\label{lem:aux}
Both horizontal compositions in \eqref{eq:composed10} are fully faithful.
\end{lemma}
\begin{proof}
Let us start by showing that both horizontal compositions in the diagram
\begin{equation}\label{diagram-new}
\xymatrix{
\Ab(k)_\bbQ \ar[r]^-{\eqref{eq:equivalence}^{-1}} \ar@{=}[d] & \Homo(k)_\bbQ \ar[d] \ar[r]^-\mu & \Homo(k)_\bbQ/_{\!\!-\otimes \bbQ(1)} \ar[d] \\
\Ab(k)_\bbQ \ar[r]_-{\eqref{eq:equivalence}^{-1}} & \Num(k)_\bbQ \ar[r]_-\mu & \Num(k)_\bbQ/_{\!\!-\otimes \bbQ(1)}
}
\end{equation}
are fully faithful. Since $\mu$ is faithful and the middle vertical functor (and hence also the right vertical one) is full, it suffices to show that the upper horizontal composition is full. The ``curved'' functor $H^\ast_{dR}$ of diagram \eqref{eq:diagram} is faithful and symmetric monoidal. Moreover, it maps $\bbQ(j)$ to the field $k$ placed in degree $-2j$. Hence, one obtains the following inclusion
$$ \Hom_{\Homo(k)_\bbQ}(\underline{\pi}^1 M(C), \underline{\pi}^1 M(C')(j)) \hookrightarrow \Hom_{\mathrm{GrVect}(k)}(H^1_{dR}(C), H^1_{dR}(C')[-2j])$$
for every integer $j$ and smooth projective curves $C$ and $C'$. The right-hand-side vanishes for $j \neq 0$ and consequently also the left-hand-side. Cleary the same holds for all the direct factors of $\underline{\pi}^1M(C)$ and $\underline{\pi}^1M(C')(j)$. By definition of the orbit category one then concludes that the upper horizontal composition is full. This proves our claim. Now, since $R_H$ is full and $R_\cN$ is fully faithful (see Theorem~\ref{thm:bridge}) the proof follows from the commutativity of diagram \eqref{eq:composed10}.
\end{proof}
{\bf Item (ii)} By construction, the equivalence \eqref{eq:equivalence} is compatible with de Rham cohomology in the sense that the following diagram commutes:
\begin{equation}\label{eq:H11}
\xymatrix{
\{\underline{\pi}^1M(C)\,|\, C \in \Curv(k)\}^\natural \subset \Homo(k) \ar[d]^-{\simeq}_-{\eqref{eq:equivalence}} \ar[rr]^-{\oplus_{i\, \mathrm{odd}}H^i_{dR}(-)} && \mathrm{Vect}(k) \\
\Ab(k)_\bbQ \ar@/_2ex/[urr]_-{H^1_{dR}(-)} &&\,.
}
\end{equation}
Recall from the proof of Theorem~\ref{thm:bridge} that the functor \eqref{eq:comp} agrees with the super-perioditization \eqref{eq:description} of de Rham cohomology. By combining this fact with the commutativity of diagram \eqref{eq:dig-1}, one concludes that
\begin{equation*}
\Homo(k)_\bbQ \stackrel{\mu}{\too} \Homo(k)_\bbQ/_{\!\!-\otimes \bbQ(1)} \stackrel{R_H}{\too} \NHomo(k)_\bbQ \stackrel{HP^-}{\too} \mathrm{Vect}(k)
\end{equation*}
agrees with the functor $\oplus_{i\, \mathrm{odd}}H^i_{dR}(-)$. Hence, since by construction the horizontal compositions in \eqref{eq:composed10}-\eqref{eq:aux33} map ${\bf J}(N)$ to $\tau(N)$, one obtains the equality 
\begin{equation}\label{eq:equality-new}
H^1_{dR}({\bf J}(N)) = HP^-(\tau(N))\,.
\end{equation}
Now, recall from item (i) that $N$ decomposes (uniquely) into a finite direct sum $S_1 \oplus \cdots \oplus S_n$ of simple objects in $\NNum(k)_\bbQ$ and that $\tau(N) \in \overline{\Ab}(k)_\bbQ$ is by definition the subsume consisting of those simple objects that belong to $\cS$. One has then an inclusion $\delta: \tau(N) \hookrightarrow N$ and a projection $N \twoheadrightarrow \tau(N)$ morphism such that $\rho \circ \delta = \id_{\tau(N)}$. Let $\overline{\delta}$ be a lift of $\delta$ and $\overline{\rho}$ a lift of $\rho$ along the vertical functor of diagram \eqref{eq:aux33}. Note that $\overline{\rho} \circ \overline{\delta} = \id_{\tau(N)}$ since $\tau(N) \in \overline{\Ab}(k)_\bbQ$.

The following equivalence of categories
\begin{equation}\label{eq:induced}
\Ab(k)_\bbQ \simeq \{\underline{\pi}^1M(C)\,|\,C \in \Curv(k)\}^\natural \stackrel{R_\cN \circ \mu}{\too} \overline{\Ab}(k)_\bbQ
\end{equation}
show us that every object in $\overline{\Ab}(k)_\bbQ$ is a direct factor of a noncommutative numerical motive of the form $\underline{\pi}^1 \perf(C)$ with $C \in \Curv(k)$; note that thanks to the commutativity of diagram \eqref{eq:big} the image of $M(C)$ under \eqref{eq:induced} identifies with $\perf(C)$. Hence, since by construction $\tau(N) \in \overline{\Ab}(k)_\bbQ$ there exists a smooth projective curve $C_N$, a surjective morphism $\overline{\Gamma_N}: \underline{\pi}^1 \perf(C_N) \twoheadrightarrow \tau(N)$ in $\overline{\Ab}(k)_\bbQ$, and a section $s:\tau(N) \to \underline{\pi}^1\perf(C_N)$ of $\overline{\Gamma_N}$ expressing $\tau(N)$ as a direct factor of $\underline{\pi}^1\perf(C_N)$. Consider now the following composition
\begin{equation}\label{eq:morph1}
\underline{\pi}^1 \perf(C_N) \stackrel{\overline{\Gamma_N}}{\twoheadrightarrow} \tau(N) \stackrel{\overline{\delta}}{\hookrightarrow} N
\end{equation}
in $\NHomo(k)_\bbQ$. Since $\underline{\pi}^1\perf(C_N)$ is a direct factor of $\perf(C_N)$, the morphism \eqref{eq:morph1} can be extended to $\perf(C_N)$ (by mapping the other direct sum component to zero) and furthermore lifted to a morphism $\Gamma_N:\perf(C_N) \to N$ in $\NChow(k)_\bbQ$ (using the fact that the functor $\NChow(k)_\bbQ \to \NHomo(k)_\bbQ$ is full).

Let us now show that $HP^-(\tau(N)) = \mathrm{Im}HP^-(\Gamma_N)$. Recall that the odd de Rham cohomology of a curve is supported in degree $1$. As a consequence, making use again of the commutativity of diagram \eqref{eq:big}, one obtains the commutative square
$$
\xymatrix{
H^1_{dR}(C_N) = HP^-(\perf(C_N)) \ar@{=}[d] \ar[rr]^-{HP^-(\Gamma_N)} && HP^-(N) \\
H^1_{dR}(C_N) = HP^-(\underline{\pi}^1\perf(C_N))  \ar[rr]_-{HP^-(\overline{\Gamma_N})} && HP^-(\tau(N)) \ar[u]_-{HP^-(\overline{\delta})}
}
$$
which shows us that the morphism $HP^-(\Gamma_N)$ factors through $HP^-(\tau(N))$. Since $\overline{\Gamma_N}$ admits a section $s$ and $\overline{\delta}$ a retraction $\overline{\rho}$, one concludes by functoriality that $HP^-(\overline{\Gamma_N})$ is surjective and that $HP^-(\overline{\delta})$ is injective. This implies the equality $HP^-(\tau(N)) = \mathrm{Im} HP^-(\Gamma_N)$. Finally, by combining this equality with \eqref{eq:equality-new} one obtains the desired equality $H^1_{dR}({\bf J}(N)) = \mathrm{Im}HP^-(\Gamma_N)$.

\vspace{0.1cm}

{\bf Item (iii)} 
As explained at item (ii), one has $H^1_{dR}({\bf J}(N)) = HP^-(\tau(N)) \subseteq HP^-_{\mathrm{curv}}(N)$. Hence, it remains only to prove the converse inclusion $HP^-_{\mathrm{curv}}(N) \subseteq HP^-(\tau(N))$. Recall that for every smooth projective curve $C$, one has a canonical (split) inclusion morphism $\underline{\pi}^1M(C) \hookrightarrow M(C)$ in $\Homo(k)_\bbQ$. Let us denote by $\iota: \underline{\pi}^1\perf(C) \hookrightarrow \perf(C)$ its image in $\NHomo(k)_\bbQ$ under the composed functor
\begin{equation}\label{eq:composed1}
\Homo(k)_\bbQ \stackrel{\mu}{\too} \Homo(k)_\bbQ/_{\!\!-\otimes \bbQ(1)} \stackrel{R_\cN}{\too} \NHomo(k)_\bbQ\,.
\end{equation}
Given a morphism $\Gamma: \perf(C) \to N$ in $\NChow(k)_\bbQ$, one can then consider the following diagram
\begin{equation}\label{eq:square}
\xymatrix{
\perf(C) \ar[rr]^-\Gamma && N \\
\underline{\pi}^1 \perf(C) \ar[u]^-{\iota} \ar@{-->}[rr]_-{\overline{\Gamma}} && \tau(N) \ar[u]_-\delta
}
\end{equation}
in $\NNum(k)_\bbQ$. Since $\underline{\pi}^1 \perf(C)\in \overline{\Ab}(k)_\bbQ$ one observes that by definition of $\tau(N)$ the morphism $\Gamma \circ \iota$ factors uniquely through $\tau(N)$ via a morphism $\overline{\Gamma}$. Diagram \eqref{eq:square} becomes then a well-defined commutative square. Now, let us consider the following diagram
\begin{equation}\label{eq:square1}
\xymatrix{
\perf(C) \ar[rr]^-\Gamma && N \\
\underline{\pi}^1 \perf(C) \ar[u]^-{\iota} \ar[rr]_-{\overline{\Gamma}} && \tau(N) \ar[u]_-{\overline{\delta}}
}
\end{equation}
in $\NHomo(k)_\bbQ$. Note that \eqref{eq:square1} is mapped to \eqref{eq:square} by the functor $\NHomo(k)_\bbQ \to \NNum(k)_\bbQ$. As in the ``commutative world'', one has $\perf(C) \simeq \perf(C)^\op$. Hence, since the conjecture $D_{NC}(\perf(C) \otimes N)$ holds one has the following equality of $\bbQ$-vector spaces
$$ \Hom_{\NHomo(k)_\bbQ}(\perf(C),N) = \Hom_{\NNum(k)_\bbQ}(\perf(C),N)\,;$$
see \cite[\S10]{Galois}. Clearly, the same holds with $\perf(C)$ replaced by $\underline{\pi}^1\perf(C)$. Consequently, one concludes that the above square \eqref{eq:square1} is in fact commutative. By applying it with the functor $HP^-$ (and using again the commutativity of diagram \eqref{eq:big}), one obtains the commutative square
$$
\xymatrix{
H^1_{dR}(C) = HP^-(\perf(C)) \ar[rr]^-{HP^-(\Gamma)} && HP^-(N) \\
H^1_{dR}(C) = HP^-(\underline{\pi}^1\perf(C))  \ar@{=}[u] \ar[rr]_-{HP^-(\overline{\Gamma})}&& HP^-(\tau(N)) \ar[u]_-{HP^-(\overline{\delta})}\,,
}
$$
which shows us that the morphism $HP^-(\Gamma)$ factors through $HP^-(\tau(N))$. Since this factorization occurs for every smooth projective curve $C$ and for every morphism $\Gamma: \perf(C) \to N$ in $\NChow(k)_\bbQ$ one obtains the inclusion $HP^-(\tau(N)) \subseteq HP^-_{\mathrm{curv}}(N)$. This concludes the proof of item (iii) and hence of Theorem~\ref{thm:main}.
\section{Bilinear pairings}
In this section one proves some (technical) results concerning bilinear pairings. These results will play a key role in the proof of Theorem~\ref{thm:main2}.
\begin{proposition}\label{prop:sum}
For every smooth projective $k$-scheme $X$ of dimension $d$ there is a natural isomorphism of $k$-vector spaces
\begin{equation}\label{eq:can-iso}
HP^-_{\mathrm{curv}}(\perf(X)) \simeq \oplus_{i=0}^{d-1} NH_{dR}^{2i+1}(X)\,.
\end{equation}
\end{proposition}
\begin{proof}
As explained in the proof of item (ii) of Theorem~\ref{thm:main}, the composition 
$$ \Chow(k)_\bbQ \stackrel{\mu}{\too} \Chow(k)_\bbQ/_{\!\!-\otimes \bbQ(1)} \stackrel{R}{\too} \NChow(k)_\bbQ \stackrel{HP^-}{\too} \mathrm{Vect}(k)$$
agrees with the functor $\oplus_{i\,\mathrm{odd}} H^i_{dR}(-)$. The commutativity of diagram \eqref{eq:big}, combined with the fact that $R$ is fully faithful and with the fact that the odd de Rham cohomology of a curve is supported in degree one, implies that the $k$-vector space $HP^-_{\mathrm{curv}}(\perf(X))$ identifies with 
\begin{equation}\label{eq:ident1}
\sum_{C,\gamma} \mathrm{Im} \big(H^1_{dR}(C) \stackrel{H^1_{dR}(\gamma)}{\too} \oplus_{i \,\mathrm{odd}} H_{dR}^i(X)\big)\,,
\end{equation}
where $\gamma$ is an element of 
$$ \Hom_{\Chow(k)_\bbQ/_{\!\!-\otimes\bbQ(1)}}(M(C),M(X)) = \oplus_{i=-1}^d\Hom_{\Chow(k)_\bbQ}(M(C),M(X)(i))\,.$$ 
Note that due to dimensional reasons the morphisms $H^1_{dR}(\gamma_{-1})$ and $H^1_{dR}(\gamma_d)$ are zero. Hence, since $\oplus_{i\,\mathrm{odd}}H^i_{dR}(-)=\oplus_{i=0}^{d-1}H_{dR}^{2i+1}(-)$, the above sum \eqref{eq:ident1} identifies furthermore with 
\begin{equation}\label{eq:ident2}
\sum_{C,\{\gamma_i\}_{i=0}^{d-1}} \mathrm{Im}\big(H^1_{dR}(C) \stackrel{\oplus_{i=0}^{d-1} H^1_{dR}(\gamma_i)}{\too} \oplus_{i=0}^{d-1} H_{dR}^{2i+1}(X)\big)
\end{equation}
where $\gamma_i : M(C) \to M(X)(i)$ is a morphism in $\Chow(k)_\bbQ$. Clearly this latter sum identifies with
\begin{equation}\label{eq:ident3}
\oplus_{i=0}^{d-1}\big(\sum_{C,\gamma_i} \mathrm{Im}(H^1_{dR}(C) \stackrel{H^1_{dR}(\gamma_i)}{\too} H^{2i+1}_{dR}(X)) \big)
\end{equation}
and so by combining \eqref{eq:ident1}-\eqref{eq:ident3} one obtains the natural isomorphism \eqref{eq:can-iso}. 
\end{proof}
\subsection*{Betti cohomology}
In this subsection one assumes that $k$ is a subfield of $\bbC$. Let $X$ be a smooth projective $k$-scheme of dimension $d$. Similarly to de Rham cohomology, one introduces the $\bbQ$-vector spaces
\begin{eqnarray*}
NH_B^{2i+1}(X) := \sum_{C,\gamma_i} \mathrm{Im}(H^1_B(C) \stackrel{H^1_B(\gamma_i)}{\too} H_B^{2i+1}(X)) && 0 \leq i \leq d-1\,,
\end{eqnarray*}
where $C \in \Curv(k)$ and $\gamma_i \in \Hom_{\Chow(k)_\bbQ}(M(C),M(X)(i))$.
\begin{lemma}\label{lem:Groth}
There are natural isomorphisms of $\bbC$-vector spaces
\begin{eqnarray*}
NH_{dR}^{2i+1}(X) \otimes_k \bbC \simeq NH_B^{2i+1}(X) \otimes_\bbQ \bbC && 0 \leq i \leq d-1\,.
\end{eqnarray*}
\end{lemma}
\begin{proof}
Note that 
\begin{eqnarray*}
NH_{dR}^{2i+1}(X) \otimes_k \bbC & \simeq&  \sum_{C,\gamma_i} \mathrm{Im}(H^1_{dR}(C)_\bbC \stackrel{H^1_{dR}(\gamma_i)_\bbC}{\to} H_{dR}^{2i+1}(X)_\bbC ) \\
NH_{B}^{2i+1}(X) \otimes_\bbQ \bbC & \simeq& \sum_{C,\gamma_i} \mathrm{Im}(H^1_{B}(C)_\bbC \stackrel{H^1_{B}(\gamma_i)_\bbC}{\to} H_{B}^{2i+1}(X)_\bbC)\,,
\end{eqnarray*}
where $C \in \Curv(k)$ and $\gamma_i \in \Hom_{\Chow(k)_\bbQ}(M(C),M(X)(i))$. The proof follows now from the naturality of the comparison isomorphism
\begin{eqnarray}\label{eq:comparison}
H^\ast_{dR}(Y)\otimes_k \bbC \stackrel{\simeq}{\too} H^\ast_B(Y) \otimes_\bbQ \bbC && Y \in \SmProj(k)
\end{eqnarray}
established by Grothendieck in \cite{Grothendieck}.
\end{proof}
Similarly to de Rham cohomology, one has also intersection bilinear pairings
\begin{eqnarray}\label{eq:pairings-B}
\langle-,-\rangle: NH_B^{2d-2i-1}(X) \times NH_B^{2i+1}(X) \too \bbQ && 0 \leq i \leq d-1\,.
\end{eqnarray}
\begin{proposition}\label{prop:agree}
The pairings \eqref{eq:pairings1} are non-degenerate if and only if the pairings \eqref{eq:pairings-B} are non-degenerate.
\end{proposition}
\begin{proof}
Grothendieck's comparison isomorphism \eqref{eq:comparison} is compatible with the intersection pairings and so one obtains the following commutative diagram:
$$
\xymatrix{
NH_B^{2d-2i-1}(X)_\bbC \times NH_B^{2i+1}(X)_\bbC \ar[rr]^-{\langle-,-\rangle} && \bbC \\
NH_{dR}^{2d-2i-1}(X)_\bbC \times NH_{dR}^{2i+1}(X)_\bbC \ar[rr]_-{\langle -,-\rangle_\bbC} \ar[u]_-\simeq^-{\eqref{eq:comparison}} && \bbC \ar[u]^-\simeq_-{\eqref{eq:comparison}}\,.
}
$$
The proof now follows from the general Lemma~\ref{lem:general} below applied to the intersection pairings on de Rham cohomology ($k \subset \bbC$) and Betti cohomology ($\bbQ\subset\bbC$).
\end{proof}
\begin{lemma}\label{lem:general}
Let $F$ be a field, $V$ and $W$ two finite dimensional $F$-vector spaces, $\langle -,-\rangle :V \times W \to F$ a bilinear pairing, and $F \subseteq L$ a field extension. Then, the pairing $\langle-,-\rangle$ is non-degenerate if and only if the pairing $\langle-,-\rangle_L:V_L \times W_L \to L$ is non-degenerate.
\end{lemma}
\begin{proof}
Choose a basis $\{v_i\}_{i=1}^n$ of $V$ and a basis $\{w_j\}_{j=1}^m$ of $W$. Under such choices, the non-degeneracy of $\langle-,-\rangle$ reduces to the following two conditions:
\begin{itemize}
\item[(i)] for every $i \in \{1,\ldots, n\}$ there exists a $i'\in \{1, \ldots, m\}$ such that $\langle v_i,w_{i'}\rangle \neq 0$;
\item[(ii)]  for every $j \in \{1,\ldots, m\}$ there exists a $j'\in \{1, \ldots, n\}$ such that $\langle v_{j'},w_j\rangle \neq 0$.
\end{itemize}
Hence, the proof follows from the fact that  $\{v_i \otimes_F L\}_{i=1}^n$ (resp. $\{w_j \otimes_F L\}_{j=1}^m$) is a basis of $V_L$ (resp. of $W_L$) and that
$\langle v_i \otimes_F L, w_j \otimes_F L \rangle =\langle v_i, w_j \rangle$ for every $i \in \{1, \ldots, n\}$ and $j \in \{1, \ldots, m\}$.
\end{proof}

\begin{remark}\label{rk:explanation}
As proved by Vial in \cite[Lemma~2.1]{Vial}, the above pairings \eqref{eq:pairings-B} with $i=0$ and $i=d-1$ are always non-degenerate. Moreover, the non-degeneracy of the remain cases follows from Grothendieck's standard conjecture of Lefschetz type. Hence, the pairings \eqref{eq:pairings-B} are non-degenerate for curves, surfaces, abelian varieties, complete intersections, uniruled threefolds, rationally connected fourfolds, and for any smooth hypersurface section, product, or finite quotient thereof. Making use of Proposition \ref{prop:agree} one then observes that the same holds for the pairings \eqref{eq:pairings1}.
\end{remark}
\section{Proof of theorem~\ref{thm:main2}}
Similarly to de Rham cohomology, Betti cohomology also gives rise to a well-defined $\otimes$-functor $H_B^\ast: \Homo(k)_\bbQ \to \mathrm{GrVect}(\bbQ)$. As proved in Proposition \ref{prop:agree}, the intersection pairings \eqref{eq:pairings1} are non-degenerate if and only if the intersection pairings \eqref{eq:pairings-B} on Betti cohomology are non-degenerate. Hence, following \cite[Thm.~2]{Vial}, there exists mutually orthogonal idempotents $\Pi_{2i+1} \in \End_{\Homo(k)_\bbQ}(M(X))$, $0 \leq i \leq d-1$, with the following two properties:
\begin{eqnarray*}
\Pi_{2i+1}M(X) \simeq \underline{\pi}^1M(J_i^a(X))(-i) &\mathrm{and}& H^\ast_B(\underline{\pi}^1M(J^a_i(X))(-i))\simeq NH_B^{2i+1}(X)\,. 
\end{eqnarray*}
As a consequence, the direct sum 
\begin{equation}\label{eq:direct-sum}
\oplus_{i=0}^{d-1} \Pi_{2i+1}M(X) \simeq \oplus_{i=0}^{d-1} \underline{\pi}^1 M(J_i^a(X))(-i)
\end{equation}
is the direct factor of $M(X)$ associated to the idempotent $\sum_{i=0}^{d-1}\Pi_{2i+1}$. Now, recall from Theorem~\ref{thm:bridge} that one has the following commutative diagram:
\begin{equation}\label{eq:dig-11}
\xymatrix{
\Homo(k)_\bbQ\ar[r]^-\mu \ar[d] & \Homo(k)_\bbQ/_{\!\!-\otimes \bbQ(1)} \ar[d] \ar[r]^-{R_H} & \NHomo(k)_\bbQ \ar[d] \\
\Num(k)_\bbQ \ar[r]^-\mu & \Num(k)_\bbQ/_{\!\!-\otimes \bbQ(1)} \ar[r]^-{R_\cN} & \NNum(k)_\bbQ\,.
}
\end{equation}
By combining the universal property of the orbit category with the commutativity of diagram \eqref{eq:big}, one then observes that the image of \eqref{eq:direct-sum} (in $\NNum(k)_\bbQ$) under the composed functor \eqref{eq:dig-11} identifies with $\oplus_{i=0}^{d-1} \underline{\pi}^1 \perf(J^a_i(X))$. As explained in \cite[Prop.~4.3.4.1]{Andre} one has an equivalence
\begin{equation}\label{eq:equivalence-new}
\{\underline{\pi}^1 M(C)\,|\, C \in \Curv(k)\}^\natural \simeq \{\underline{\pi}^1M(X)\,|\,X \in \SmProj(k)\}
\end{equation}
of subcategories of $\Homo(k)_\bbQ$. Therefore, since $\overline{\Ab}(k)_\bbQ$ is closed under direct sums and $\underline{\pi}^1\perf(J^a_i(X))$ identifies with the image of $\underline{\pi}^1 M(J^a_i(X))$ along the composed functor \eqref{eq:dig-11}, one concludes that $\oplus_{i=0}^{d-1}\underline{\pi}^1\perf(J_i^a(X))$ belongs to $\overline{\Ab}(k)_\bbQ$. Now, recall from the proof of Theorem~\ref{thm:main} that besides this direct factor of $\perf(X)$ one has also the noncommutative numerical motive $\tau(\perf(X)) \in \overline{\Ab}(k)_\bbQ$. By definition of $\tau(\perf(X))$ one observes that there exists a split surjective morphism 
\begin{equation}\label{eq:split}
\tau(\perf(X)) \twoheadrightarrow \oplus^{d-1}_{i=0} \underline{\pi}^1\perf(J^a_i(X))
\end{equation}
in $\overline{\Ab}(k)_\bbQ$. Let us now prove that \eqref{eq:split} is an isomorphism. Consider the following composition with values in the category of finite dimensional super $\bbC$-vector spaces
\begin{equation}\label{eq:composed}
\overline{\Ab}(k)_\bbQ \subset \NHomo(k)_\bbQ \stackrel{HP^\pm}{\too} \mathrm{sVect}(k) \stackrel{-\otimes_k \bbC}{\too} \mathrm{sVect}(\bbC)\,.
\end{equation}
Since $HP^\pm$ and $-\otimes_k\bbC$ are faithful, the composition \eqref{eq:composed} is also faithful. Therefore, since the above surjective morphism \eqref{eq:split} admits a section, it suffices to prove the following inequality
\begin{equation}\label{eq:inequality}
\mathrm{dim}(HP^\pm(\tau(\perf(X)))\otimes_k \bbC) \leq \mathrm{dim}(HP^\pm(\oplus_{i=0}^{d-1} \underline{\pi}^1\perf(J^a_i(X)))\otimes_k \bbC) \,.
\end{equation}
On one hand one has:
\begin{eqnarray}
HP^\pm(\tau(\perf(X)))\otimes_k \bbC & = & HP^-(\tau(\perf(X))) \otimes_k \bbC \label{is1} \\
& \subseteq & HP^-_{\mathrm{curv}}(\perf(X)) \otimes_k \bbC \label{is2} \\
& \simeq & \oplus_{i=0}^{d-1} NH_{dR}^{2i+1}(X) \otimes_k \bbC \label{is3}  \\
& \simeq & \oplus_{i=0}^{d-1} NH_B^{2i+1}(X) \otimes_\bbQ \bbC \label{is4}\,.
\end{eqnarray}
Equality \eqref{is1} follows from the fact that $\tau(\perf(X)) \in \overline{\Ab}(k)_\bbQ$ and from the description \eqref{eq:description} of the composition \eqref{eq:comp} as the super-perioditization of de Rham cohomology. As explained in the proof of Theorem~\ref{thm:main}, we have the equality $H^1_{dR}({\bf J}(\perf(X))=HP^-(\tau(\perf(X)))$. Hence, inclusion \eqref{is2} follows from item (ii) of Theorem~\ref{thm:main}. Finally, isomorphism \eqref{is3} is the content of Proposition~\eqref{prop:sum} and isomorphism \eqref{is4} follows from Lemma~\ref{lem:Groth}. On the other hand one has:
\begin{eqnarray}
HP^\pm(\oplus_{i=0}^{d-1} \underline{\pi}^1\perf(J_i^a(X))) \otimes_k \bbC & \simeq & \oplus_{i=0}^{d-1}HP^\pm (\underline{\pi}^1 \perf(J_i^a(X))) \otimes_k \bbC \nonumber \\
& \simeq & \oplus_{i=0}^{d-1}H^1_{dR}(M(J_i^a(X))(-i)) \otimes_k \bbC \label{isom:2} \\
& \simeq & \oplus_{i=0}^{d-1}H^1_{B}(M(J_i^a(X))(-i))  \otimes_\bbQ \bbC \label{isom:3}  \\
& \simeq & \oplus_{i=0}^{d-1} NH_B^{2i+1}(X) \otimes_\bbQ \bbC\,. \label{isom:4}
\end{eqnarray}
Isomorphism \eqref{isom:2} follows from the fact that the composed functor \eqref{eq:comp} agrees with the super-perioditization \eqref{eq:description} of de Rham cohomology and that $\underline{\pi}^1\perf(J^a_i(X))$ is the image of $\underline{\pi}^1 M(J^a_i(X))(-i)$ under the composed functor \eqref{eq:dig-11}. Isomorphism \eqref{isom:3} follows from \eqref{eq:comparison} and finally \eqref{isom:4} follows from the above isomorphism $H^\ast_B(\underline{\pi}^1M(J_i^a(X))(-i))\simeq NH_B^{2i+1}(X)$. These isomorphisms imply the above inequality \eqref{eq:inequality} and consequently that \eqref{eq:split} is an isomorphism.

Now, by combining equivalences \eqref{eq:equivalence} and \eqref{eq:equivalence-new}, one observes that the passage from a smooth projective $k$-scheme to its Picard (abelian) variety  gives also rise to an equivalence of categories
\begin{eqnarray}\label{eq:new-equivalence}
& \Num(k)_\bbQ \supset \{\underline{\pi}^1M(X)\,|\, X \in \SmProj(k)\} \stackrel{\sim}{\too} \Ab(k)_\bbQ & X \mapsto \mathrm{Pic}^0(X)\,.
\end{eqnarray}
The direct sum $\oplus_{i=0}^{d-1}\underline{\pi}^1M(J_i^a(X))$ is mapped to $\oplus_{i=0}^{d-1}\underline{\pi}^1\perf(J_i^a(X))$ under the composed functor \eqref{eq:dig-11}. Hence, by the construction of the Jacobian functor one concludes that ${\bf J}(\perf(X))$ identifies with the image of $\oplus_{i=0}^{d-1}\underline{\pi}^1M(J_i^a(X))$ under the above equivalence \eqref{eq:new-equivalence}. The desired isomorphism \eqref{eq:isom-last} follows now from the natural isomorphism 
$$\oplus_{i=0}^{d-1} \underline{\pi}^1 \perf(J^a_i(X)) \simeq \underline{\pi}^1 M(\cup_{i=0}^{d-1} J_i^a(X))$$ and from the fact that the algebraic variety $J_i^a(X)$ agree with their own Picard variety. Finally, the above arguments show us that
\begin{equation}\label{eq:iso-last}
HP^-(\tau(\perf(X)))\otimes_k \bbC \simeq \oplus_{i=0}^{d-1} NH_{dR}^{2i+1}(X) \otimes_k \bbC\,.
\end{equation}
Hence, by combing equality \eqref{eq:equality-new} (with $N=\perf(X)$) with \eqref{eq:iso-last} one obtains the isomorphism $H^1_{dR}({\bf J}(\perf(X))) \simeq \oplus_{i=0}^{d-1} NH_{dR}^{2i+1}(X) \otimes_k \bbC$.


\begin{thebibliography}{00}

\bibitem{Andre} Y.~Andr{\'e}, {\em Une introduction aux motifs (motifs purs, motifs mixtes, p{\'e}riodes)}. Panoramas et Synth{\`e}ses {\bf 17}. Soci{\'e}t{\'e} Math{\'e}matique de France, Paris, 2004.
 
\bibitem{CT1} D.-C.~Cisinski and G.~Tabuada, {\em Symmetric monoidal structure on nonommutative motives}. Journal of $K$-Theory {\bf 9} (2012), no.~2, 201--268.

\bibitem {Deligne} P.~Degline, {\em Hodge cycles on abelian varieties}. Notes by J.~S.~Milne of the seminar ``P{\'e}riodes des Int{\'e}grales Ab{\'e}liennes'' given by P.~Deligne at IHES, 1978-79. Version of July 4 (2003) available at \url{http://www.jmilne.org/math/Documents/Deligne82.pdf}.

\bibitem{Griffiths} P.~Griffiths, {\em On the periods of certain rational integrals I, II}. Ann. of Math. (2) {\bf 90} (1969), 460--495.

\bibitem{Grothendieck} A.~Grothendieck, {\em On the de Rham cohomology of algebraic varieties}. Inst. Hautes {\'E}tudes Sci. Publ. Math. no. {\bf 29} (1966), 95--103.

\bibitem{ICM-Keller} B.~Keller, {\em On differential graded
    categories}. International Congress of Mathematicians (Madrid), Vol.~II,
  151--190, Eur.~Math.~Soc., Z{\"u}rich, 2006.

\bibitem{IAS} M.~Kontsevich, {\em Noncommutative motives}. Talk at the Institute for Advanced Study on the occasion of the $61^{\mathrm{st}}$ birthday of Pierre Deligne, October 2005. Video available at \url{http://video.ias.edu/Geometry-and-Arithmetic}.    

\bibitem{ENS} \bysame, {\em Triangulated categories and geometry}. Course at the {\'E}cole Normale Sup{\'e}rieure, Paris, 1998. Notes available at \url{www.math.uchicago.edu/mitya/langlands.html}
    
\bibitem{Miami} \bysame, {\em Mixed noncommutative motives}. Talk at the Workshop on Homological Mirror Symmetry. University of Miami. 2010. Notes available at \url{www-math.mit.edu/auroux/frg/miami10-notes}.  

\bibitem{finMot} \bysame, {\em Notes on motives in finite characteristic}.  Algebra, arithmetic, and geometry: in honor of Yu. I. Manin. Vol. II,  213--247, Progr. Math., {\bf 270}, BirkhŠuser Boston, MA, 2009.    
    
\bibitem{LO} V.~Lunts and D.~Orlov, {\em Uniqueness of enhancement for triangulated categories}. J. Amer. Math. Soc. {\bf 23} (2010), no.~3, 853--908. 

\bibitem{Numerical} M.~Marcolli and G.~Tabuada, {\em Noncommutative motives, numerical equivalence, and semi-simplicity}. Available at arXiv:1105.2950. To appear in American Journal of Mathematics.     

\bibitem{Kontsevich} \bysame, {\em Kontsevich's noncommutative numerical motives}. Available at arXiv:0302013. To appear in Compositio Mathematica.

\bibitem{Galois} \bysame, {\em Noncommutative numerical motives, Tannakian structures, and motivic Galois groups}. Available at arXiv:1110.2438.
            
\bibitem{Scholl} A.~Scholl, {\em Classical motives}. Motives (Seattle, WA, 1991), 163--187, 
Proc. Sympos. Pure Math. {\bf 55}, Part 1, Amer. Math. Soc., Providence, RI, 1994.                              

\bibitem{survey} G.~Tabuada, {\em A guided tour through the garden of noncommutative motives}. Clay Mathematics Proceedings, Volume {\bf 16} (2012), 259--276.

\bibitem{CvsNC} \bysame, {\em Chow motives versus noncommutative motives}. 
Available at arXiv:1103.0200. To appear in Journal of Noncommutative Geometry.  

\bibitem{regular} \bysame, {\em $E_n$-regularity implies $E_{n-1}$-regularity}. Available at arXiv:0608.087. 

\bibitem{Vial} C.~Vial, {\em Projectors on the intermediate algebraic Jacobians}. Available at arXiv:0907.3539. 

\bibitem{Weil} A.~Weil, {\em Vari{\'e}t{\'e}s ab{\'e}liennes et courbes algŽbriques}. (French) 
Actualit{\'e}s Sci. Ind., {\bf 1064}, Publ. Inst. Math. Univ. Strasbourg {\bf 8} (1946). Hermann and Cie., Paris, 1948.

\end{thebibliography}
\end{document}